\documentclass[11pt]{article}
\usepackage{amsmath,latexsym,amssymb,amsfonts,amsbsy, amsthm}
\usepackage{graphicx}
\usepackage[pagewise]{lineno}
\usepackage{ulem}
\def\inte#1{
\displaystyle\mathop{#1\kern0pt}^\circ }

\let\grad\nabla

\def\virgp{\raise 2pt\hbox{,}}
\def\cdotpv{\raise 2pt\hbox{;}}

\def\C{\mathop{\bf C\kern 0pt}\nolimits}
\def\DD{\mathop{\bf D\kern 0pt}\nolimits}
\def\K{\mathop{\bf K\kern 0pt}\nolimits}
\def\N{\mathop{\bf N\kern 0pt}\nolimits}
\def\Q{\mathop{\bf Q\kern 0pt}\nolimits}
\def\R{\mathop{\bf R\kern 0pt}\nolimits}
\def\SS{\mathop{\bf S\kern 0pt}\nolimits}
\def\ZZ{\mathop{\bf Z\kern 0pt}\nolimits}
\def\TT{\mathop{\bf T\kern 0pt}\nolimits}

\newcommand{\beq}{\begin{equation}}
\newcommand{\eeq}{\end{equation}}
\newcommand{\ben}{\begin{eqnarray}}
\newcommand{\een}{\end{eqnarray}}
\newcommand{\beno}{\begin{eqnarray*}}
\newcommand{\eeno}{\end{eqnarray*}}

\newtheorem{thm}{Theorem}[section]
\newtheorem{lem}{Lemma}[section]
\newtheorem{rmk}{Remark}[section]

\newtheorem{aspt}{Assumption}[section]
\renewcommand{\theequation}{\thesection.\arabic{equation}}

\begin{document}

\title{Global well-posedness of the 3D Cahn-Hilliard equations  }

\author{
 Zhenbang Li \footnote{Center for Nonlinear Studies, School of Mathematics, Northwest University, Xi'an 710069, China;  School of Science, Xi'an Technological University, Xi'an 710021, China. Email: {\tt lizbmath@nwu.edu.cn}}\and Caifeng Liu \footnote{Center for Nonlinear Studies, School of Mathematics, Northwest University, Xi'an 710069, China. Email: {\tt liucaif@163.com}.}
}

\maketitle

\begin{abstract}
The Cauchy problem of the Cahn-Hilliard equations is studied in three-dimensional space. Firstly, we construct its approximate fourth-order parabolic equation, obtaining the existence of solutions by the Aubin-Lions's compactness lemma. Furthermore, we prove the uniqueness of the solution. Then, the global well-posedness is demonstrated  by using energy estimates. At last, we consider a special case and get a better result about it.
\end{abstract}


\renewcommand{\theequation}{\thesection.\arabic{equation}}
\setcounter{equation}{0}

\section{Introduction}

In this paper, we consider the following three-dimensional Cahn-Hilliard equations:
\begin{equation}\label{EB-multi-1.1}
 \begin{cases}
 \partial_{t}u=\Delta\mu,\quad x=(x_1,x_2,x_3)\in\mathbb{R}^3,~t>0,\\
    \mu=-\Delta u+\phi(u),\quad x\in\mathbb{R}^3,~t>0
     \end{cases}
     \end{equation}
      subject to initial value condition
\begin{equation}\label{EB-multi-1.2}
u(x,0)=u_{0}(x), \quad x\in\mathbb{R}^{3}.
     \end{equation}
Here, $u(x_{1}, x_{2}, x_{3}, t)$ is the relative concentration difference of the two phases material in the mixture, $\phi$ is the derivative of the chemical potential $\Phi$. This system describe the phase separation process of binary materials like alloys (see \cite{Cahn1958}). Later, the same mathematical model is also proposed in many studies of diffusion phenomena, such as describing the competition and exclusion of biological populations \cite{Coheh1981}, river bed migration process \cite{Hazewinkel1985}, diffusion of microdroplets on solid surfaces \cite{Tayler1986}, and so on.

We start with a free energy functional of the form, given by Cahn and Hilliard \cite{Cahn1958},
\begin{equation}\label{EB-multi-1.3}
F(t)=F\big(u(x,t)\big)=\int_{\mathbb{R}^{3}}(\Phi(u)+\frac{1}{2}|\nabla u|^{2})dx,
     \end{equation}
where $\Phi(u)=\int^{u}_{0}\phi(s)ds$ is the Helmholtz free energy density. A typical example of potential $\Phi$ is of logarithmic type (see \cite{Cahn1958, Mize2004}). However, this singular potential is very often replaced by a polynomial approximation of the type $\Phi(u)=(u^2-1)^2$, which is called the double-well potential.

Recently, the research on Cahn-Hilliard equation has been very rich. Most of the existing results are carried out in the bounded domain. For example, Charles and Zheng \cite{CME1986} had taken the classical double-well potential $\phi(u)=\gamma_{2}u^{3}+\gamma_{2}u-u$ and considered the global existence or blowing up in a finite time of
the solution to the initial boundary value problem, and then they had found that the sign of $\gamma_{2}$ is crucial. If $\gamma_{2}>0$ , there is a unique global solution for any initial
data $u_{0}\in H^{2}$ and satisfying the natural boundary condition $\frac{\partial u}{\partial x}|_{x=0,L}=0$. If $\gamma_{2}<0$, then the solution must blow up in a finite time for large initial data.
Miranville and Zelik \cite{Mize2004} studied the long time behaviour of the Cahn-Hilliard equations with a singular potential, in which they were able, in two bounded space dimensions, to separate the solutions from the singular values of the potential, and then the global attract result was obtained.  Bates and  Han  \cite{Ba,Ba1} considered the Cahn-Hilliard equations with  a nonlocal potential,  they studied
the existence, uniqueness and  the long-term behavior of the solutions.
In addition, there are many
results related to the Cahn-Hilliard equations on bounded domain(see \cite{AG2015, GG2014, Liu2008, Pr2006, Racke2001} and references therein) .

On the other hand, it is worth noting that there is a few references seemed to have considered the Cahn-Hilliard equations in $\mathbb{R}^{n}$ or in unbounded domains. For example, Caffarelli and Muler researched the Cauchy problem of the Cahn-Hilliard equations under the assumptions that $\phi$ is Lipschitz continuous and equals to a constant outside a bounded interval, an $L^\infty(\mathbb{R}^n)$-a priori estimate is obtained for the solution.
Considering the Cahn-Hilliard equations under the initial data $u_{0}(x)=tanh(\frac{1}{2}x)$ and $\phi(u)=\frac{u}{2}-\frac{u^{3}}{2}$, Bricmont, Kupiainen and Taskinen \cite{J. Bricmont1999} obtained stability and time decay estimates.  Liu,  Wang and Zhao \cite{Liu 2007} studied the Cauchy problem of the Cahn-Hilliard equations with some small initial condition, which the smooth nonlinear function $\phi(u)$ satisfies a certain local growth condition at some fixed point $\bar{u}\in\mathbb{R}$ and that $\|u_{0}-\bar{u}\|_{L^{1}(\mathbb{R}^{n})}$ is suitably small, then they got the global smooth solution. Jan, and Anibal \cite{Cholewa2012} obtained the local well-posedness and global existence in $H^{1}(\mathbb{R}^{n})$, and then they gave the long-time behavior in $\mathbb{R}^{n}$($n\geq3$). Besides, there are some results on the higher order parabolic equations in the unbounded domain (see \cite{DKS2012, CB2012} and references therein).

In this paper, we will consider the global well-posedness of the 3D Cauchy problem of the Cahn--Hilliard equations \eqref{EB-multi-1.1}. Above all, we give some assumptions.
\begin{aspt}\label{aspt-1.1}
Let $\Phi(\cdot)\in \mathcal{C}^3(\mathbb{R})$, and $\phi=\Phi'$. Moreover, for any $s\in\mathbb{R}$, $\phi$ and $\Phi$ satisfy
\begin{equation}
\begin{split}\label{EB-multi-1.4}
&\Phi(s)\geq 0,\quad|\phi(s)|\leqslant C(|s|^{p}+1), \\
|\phi'(s)|\leqslant& C(|s|^{p-1}+1),
    \quad|\phi''(s)|\leqslant C(|s|^{p-2}+1),\quad 2\leq p\leq \frac{21}{5},
\end{split}
     \end{equation}
for some constants $C>0$.
\end{aspt}


Then, we give the main theorem as follows:
\begin{thm}\label{thm-main-1}
Under Assumption \ref{aspt-1.1}, let $u_0\in H^2(\mathbb{R}^3)$, then the Cahn--Hilliard equations \eqref{EB-multi-1.1} has a unique global solution $u$ on $[0, +\infty)$ such that
\begin{equation*}
u\in {\mathcal{C}}([0, \infty);H^2(\mathbb{R}^3))\cap L^2_{loc}([0, \infty);H^{4}(\mathbb{R}^3)).
\end{equation*}
\end{thm}


\begin{rmk}\label{rmk-1.1}
If taking $\phi(u)=\lambda_{1}u^{3}-\lambda_{2}u$ ($\lambda_{1}, \lambda_{2}>0$), it's easy to get its global well-posedness by using the basic energy estimates and continuity criteria. However, the result of this paper is actually for more general situations, being equivalent to a generalization of it.
\end{rmk}

In this article, we are going to obtain well-posedness of the equations (\ref{EB-multi-1.1}). The main difficulties with
the equations (\ref{EB-multi-1.1}) are that the nonlinear term may possess fractional order term  and
Poincar\'{e} inequality fails on the whole domain. To overcome these difficulties, we will
first consider the local well-posedness of the problems. Based on the properties of the bi-harmonic heat flows and using the method of continuity, we obtain the local well-posedness. After
applying Sobolev embedding theorem and
Gagiardo-Nireberg inequality to establish some
necessary uniform estimates, we prove the global well-posedness of the Cahn-Hilliard equations (\ref{EB-multi-1.1}).

The plan of the paper is as follows: In section 2, we collect some elementary
facts and inequalities which will be used later, such as basic calculus in Sobolev spaces and basic properties
of the bi-harmonic heat flow. In Section 3,  we present basic energy estimates of the equtions \eqref{EB-multi-1.1} with Assumption \ref{aspt-1.1}. Section 4 is devoted to the local well-posedness of the Cahn--Hilliard equations \eqref{EB-multi-1.1}. In Section 5, we give the proof of Theorem \ref{thm-main-1}. Finally, we will consider the special case by giving a polynomial free energy density and get a better result.

Let us complete this section with the notations we are going to use in this context.
\medbreak \noindent{\bf Notations:} We denote by $\|u\|_{L^{q}}$ (or $\|u\|_{H^{q}}$) for the $\|u\|_{L^{q}(\mathbb{R}^{3})}$ (or $\|u\|_{H^{q}(\mathbb{R}^{3})}$) ($q>1$). For $s \in \mathbb{R}$, we denote the pseudo--differential operator $\Lambda^s:=(1-\Delta)^{-\frac{s}{2}}$ with the Fourier symbol $(1+|\xi|^2)^{-\frac{s}{2}}$.The constants $C$ appearing in the paper are all different, taken the appropriate number as required.


\renewcommand{\theequation}{\thesection.\arabic{equation}}
\setcounter{equation}{0} 

\section{Preliminaries}
In this section, we recall some preliminary results that are useful throughout paper.\\
\subsection{Some calculus in Sobolev spaces}





\begin{lem}[Aubin-Lions's lemma, \cite{Simon1990}]\label{lem-2}
Assume $X\subset E\subset Y$ are Banach spaces and $X\hookrightarrow\hookrightarrow E$.
Then the following embeddings are compact:

(i)$\left\{\varphi:\varphi\in L^q([0, T]; X), \displaystyle\frac{\partial \varphi}{\partial t}\in L^1([0, T]; Y)\right\}\hookrightarrow\hookrightarrow
 L^q([0, T]; E)\quad if \quad 1\leq q\leq\infty$;

(ii)$\left\{\varphi:\varphi\in L^\infty([0, T]; X), \displaystyle\frac{\partial \varphi}{\partial t}\in L^r([0, T]; Y)\right\}\hookrightarrow\hookrightarrow
 {\mathcal{C}}([0, T]; E)\quad if \quad 1< r\leq\infty$.
\end{lem}

\begin{lem}[Calculus inequalities, \cite{Klainerman2010}]\label{lem-3}
Let $s>0$. Then the following two estimates are true:\\

(i) $\|uv\|_{H^s(\mathbb{R}^{n})}\leq C\{\|u\|_{L^\infty(\mathbb{R}^{n})}\|v\|_{H^s(\mathbb{R}^{n})}+\|u\|_{H^s(\mathbb{R}^{n})}\|v\|_{L^\infty(\mathbb{R}^{n})}\}$;\\

(ii) $\|uv\|_{H^s(\mathbb{R}^{n})}\leq C\|u\|_{H^s(\mathbb{R}^{n})}\|v\|_{H^s(\mathbb{R}^{n})} \quad \mbox{for all} \quad s>\frac{n}{2}$;\\

(iii) $\|[\Lambda^s, u]v\|_{L^2(\mathbb{R}^{n})} \leq C(\|u\|_{H^s(\mathbb{R}^{n})}\|v\|_{L^\infty(\mathbb{R}^{n})}+\|\grad u\|_{L^\infty(\mathbb{R}^{n})}\|v\|_{H^{s-1}(\mathbb{R}^{n})})$.\\

\noindent where all the constants $C$ are independent of $u$ and $v$.
\end{lem}

\begin{lem}[Sobolev's embedding theorem, \cite{L Tartar2007}]\label{lem-4}
If $u\in H^{s}(\mathbb{R}^{n})$ for $s>\frac{n}{2}$, then $u\in L^{\infty}(\mathbb{R}^{n})$, with the bound
$$ \|u\|_{L^{\infty}(\mathbb{R}^{n})}\leq C \|u\|_{H^{s}(\mathbb{R}^{n})}, $$
the constant C depending only on $s$ and $n$.
\end{lem}

\begin{lem}[Gagiardo-Nireberg-Sobolev inequality, \cite{Lawrence1988}]\label{lem-5}
If $1\leq q\leq n$ and $q^{\ast}\equiv \frac{qn}{n-q}$ is the Sobolev conjugate of $q$, then
$$\|f\|_{L^{q^{\ast}}(\mathbb{R}^{n})}\leq C_{q}\|\nabla f\|_{L^{q}(\mathbb{R}^{n})}$$
for all function $f\in C_{0}^{1}(\mathbb{R}^{n})$, the optimal constant $C_{q}$ depending only on $q$ and $n$.
Especially, when $n=3$, $q=2$, we have
$$\|f\|_{L^{6}(\mathbb{R}^{3})}\leq C\|\nabla f\|_{L^{2}(\mathbb{R}^{3})}.$$
\end{lem}

\begin{lem}[\cite{Wang2013}]\label{lem-6}
Assume $\Omega$ is bounded domain or $\Omega=\mathbb{R}^{n}$, let $1\leq p,r\leq\infty$, $j,k$ are integers, $0\leq j<k$, $\frac{j}{k}\leq \theta\leq1$, and satisfy\\

(1)$\frac{1}{q}-\frac{j}{n}=\theta(\frac{1}{p}-\frac{m}{n})+\frac{1-\theta}{r};$\\

(2)if $n>pk$, $r\leq np/(n-pk)$, and if $n\leq pk$, $r\leq\infty$;\\

(3)if $n>p(k-j)$, $\frac{nr}{n+rj}\leq q\leq\frac{np}{n-p(k-j)}$, and if $n\leq p(k-j)$, $\frac{nr}{n+rj}\leq q\leq\infty$,\\

then there exists a constant $C=C(n,k,p,r,j,\theta,\Omega)$, such that
$$\|D^{j}u\|_{L^{q}(\Omega)}\leq C\|D^{k}u\|_{L^{p}(\Omega)}^{\theta}\|u\|_{L^{r}(\Omega)}^{1-\theta}.$$

\end{lem}

\begin{lem}[\cite{Bahouri2011}, \cite{Danchin2005}]\label{lem-7}
 Let $I$ be an open interval of $\mathbb{R}$ and $F : I \rightarrow \mathbb{R}$. Let $s > 0$ and
$\sigma>0$ be the smallest integer such that $\sigma > s$. Assume that $F''$ belongs to $W^{\sigma,\infty}(I;\mathbb{R})$. Let $u, v \in H^{s}(\mathbb{R}^{3})\cap L^{\infty}(\mathbb{R}^{3})$ have values in $J\subset I$. There exists a constant $C = C(s,I,J,N)$
such that
$$\|F(u)\|_{H^{s}}\leq C(1+\|u\|_{L^{\infty}})^{\sigma}\|F''\|_{W^{\sigma,\infty}(I)}\|u\|_{H^{s}},\ if \ F(0)=0,$$
and
\begin{equation*}
\begin{split}
\|F&\circ v-F\circ u\|_{H^{s}}\leq C(1+\|u\|_{L^{\infty}}+\|v\|_{L^{\infty}})^{\sigma}\|F''\|_{W^{\sigma,\infty}(I)}\\
&\times(\|u-v\|_{H^{s}}\sup_{\tau\in[0,1]}\|v+\tau(u-v)\|_{L^{\infty}}+\|u-v\|_{L^{\infty}}
\sup_{\tau\in[0,1]}\|v+\tau(u-v)\|_{H^{s}}).
\end{split}
 \end{equation*}
\end{lem}

\subsection{Basic properties of the bi-harmonic heat flow}
Let's now recall some fundamental properties of the bi-harmonic heat flow on whole domains.
Consider the solution $u(t,x)$ to the bi-harmonic heat equations:
\begin{equation}\label{EB-multi-2.1}
 \left\{
    \begin{array}{l}
    (\partial_{t}+\Delta^{2})u=0,\ \forall~ (t,x)\in\mathbb{R}^{+}\times\mathbb{R}^{3},\\
    u|_{t=0}=u_{0}(x), \ x\in\mathbb{R}^{3},
   \end{array}
    \right.
    \end{equation}
where initial data $u_{0}\in H^{s}(\mathbb{R}^{3})$ with $s>\frac{3}{2}$. Then we have
$$\hat{u}(t,\xi)=e^{-t|\xi|^{4}}\hat{u}_{0}(\xi),\ \forall \xi=(\xi_{1},\xi_{2},\xi_{3})\in\mathbb{R}^{3},$$
where $\hat{f}(\xi):=(2\pi)^{-\frac{3}{2}}\int_{\mathbb{R}^{3}}f(x)e^{-i\xi\cdot x}dx$ for $\forall f\in L^{1}(\mathbb{R}^{3})$, which implies that
\begin{equation}\label{EB-multi-2.2}
\begin{split}
u(t,x)=e^{-t\Delta^{2}}u_{0}(x)=\sum_{\xi\in\mathbb{R}^{3}}e^{-t|\xi|^{4}}\hat{u}_{0}(\xi)e^{i\xi\cdot x},\ \forall x\in\mathbb{R}^{3},
\end{split}
 \end{equation}
for any $t>0$.
Moreover, we claim that
\begin{equation}\label{EB-multi-2.3}
\begin{split}
\|u(t,\cdot)-u_{0}(\cdot)\|_{L^{\infty}(\mathbb{R}^{3})}\rightarrow 0 \ (as \ t\rightarrow 0).
\end{split}
 \end{equation}
In effect, since $u_{0}\in H^{s}(\mathbb{R}^{3})$ with $s > \frac{3}{2}$, we find $u_{0}(x)=\sum_{\xi\in\mathbb{R}^{3}}\hat{u}_{0}(\xi)e^{i\xi\cdot x}$ for any $x\in\mathbb{R}^{3}$,
which follows from \eqref{EB-multi-2.2} that
$$u(t,x)-u_{0}(x)=\sum_{\xi\in\mathbb{R}^{3}}(e^{-t|\xi|^{4}}-1)\hat{u}_{0}(\xi)e^{i\xi\cdot x},\ \forall x\in\mathbb{R}^{3}.$$
From this, we find that for any $t > 0$, $x\in\mathbb{R}^{3}$,
 \begin{equation}\label{EB-multi-2.4}
\begin{split}
|u(t,x)-u_{0}(x)|\leq \sum_{\xi\in\mathbb{R}^{3}}|e^{-t|\xi|^{4}}-1||\hat{u}_{0}(\xi)|.
\end{split}
 \end{equation}
Since
$$\sum_{\xi\in\mathbb{R}^{3}}|e^{-t|\xi|^{4}}-1||\hat{u}_{0}(\xi)|\leq\sum_{\xi\in\mathbb{R}^{3},|\xi|\geq1}|\hat{u}_{0}(\xi)|\leq C_{s}\|u_{0}\|_{H^{s}},$$
where we have used the fact that $s > \frac{3}{2}$, Lebesgue's dominated convergence theorem ensures that
$$\lim_{t\rightarrow 0^{+}}\sum_{\xi\in\mathbb{R}^{3}}|e^{-t|\xi|^{4}}-1||\hat{u}_{0}(\xi)|=0,$$
which along with \eqref{EB-multi-2.4} yields \eqref{EB-multi-2.3}.
\begin{rmk}\label{rmk-2.1}
According to \eqref{EB-multi-2.3}, we know that, if initial data $u_{0}\in H^{s}(\mathbb{R}^{3})$ with $s > \frac{3}{2}$, then there is a positive time $T_{1}$ such that, for any
$t \in [0,T_{1}]$, there holds that
 \begin{equation}\label{EB-multi-2.5}
\begin{split}
\|e^{-t\Delta^{2}}u_{0}\|_{L^{\infty}(\mathbb{R}^{3})}\leq\|u_{0}\|_{L^{\infty}(\mathbb{R}^{3})}.
\end{split}
 \end{equation}
 Moreover, there exists a constant $M>0$, such that
  \begin{equation}\label{EB-multi-2.6}
\begin{split}
\|u_{0}\|_{L^{\infty}(\mathbb{R}^{3})}\leq M.
\end{split}
 \end{equation}

\end{rmk}

\renewcommand{\theequation}{\thesection.\arabic{equation}}
\setcounter{equation}{0}

\section{Basic energy estimates}
\begin{lem}\label{lem-3-1}
Under the assumptions in Theorem \ref{thm-main-1}, let $u$ be a smooth solution to the equations \eqref{EB-multi-1.1} on $[0,T)$ for $0<T<+\infty$, then there holds
\begin{equation}\label{EB-multi-3.1}
u\in \mathcal{C}([0,T);H^{2}(\mathbb{R}^{3}))\cap L^{2}([0,T);H^{4}(\mathbb{R}^{3})),
     \end{equation}
and moreover,
\begin{equation}\label{EB-multi-3.2}
\|u\|_{L^{\infty}([0,T);H^{2})}^{2}+\|\partial_{t}u\|_{L^{2}([0,T);L^{2})}^{2}+\|u\|_{L^{2}([0,T);H^{4})}^{2}\leq C(T,u_{0}).
     \end{equation}
\end{lem}

\begin{proof}
Derivating of $t$ in $F$ which is mentioned in \eqref{EB-multi-1.3}, we get
\begin{equation*}
\begin{split}
\frac{d F}{dt}&=\int_{\mathbb{R}^{3}}[\phi(u)\partial_{t}u+\nabla u\cdot\nabla\partial_{t}u]dx\\
&=\int_{\mathbb{R}^{3}}[\phi(u)\Delta(-\Delta u+\phi(u))-\Delta u\cdot\Delta(-\Delta u+\phi(u))]dx\\
&=\int_{\mathbb{R}^{3}}[-\nabla\phi(u)(-\nabla\Delta u+\nabla\phi(u))+\nabla\Delta u\cdot(-\nabla\Delta u+\nabla\phi(u))]dx\\
&=-\int_{\mathbb{R}^{3}}[\nabla\phi(u)-\nabla\Delta u]^{2}dx\leq 0.
\end{split}
 \end{equation*}
And owning to Assumption \ref{aspt-1.1}, we have
\begin{equation*}
\begin{split}
F(t)\leq F(0)&=\int_{\mathbb{R}^{3}}[\Phi(u_{0})+\frac{1}{2}(\nabla u_{0})^{2}]dx\\
&=\int_{\mathbb{R}^{3}}\Phi(u_{0})dx+\frac{1}{2}\|\nabla u_{0}\|_{L^{2}}^{2}\leq C(u_{0}),
\end{split}
\end{equation*}
follows that
\begin{equation}
\label{EB-multi-3.3}
\|\nabla u\|_{L^{2}}^{2}+2\int_{\mathbb{R}^{3}}\Phi(u)dx\leq F(0)\leq C(u_{0}).
\end{equation}
Taking $L^2$ inner product with $u$ to the system \eqref{EB-multi-1.1}, it gives us
 \begin{equation}\label{EB-multi-3.4}
\frac{1}{2}\frac{d}{dt}\|u\|_{L^2}^2=-\|\Delta u\|_{L^{2}}^2+\int_{\mathbb{R}^{3}}\phi(u)\Delta udx.
\end{equation}
Using \eqref{EB-multi-1.4} and H\"{o}lder's inequality,  the second term on the right of the equation (\ref{EB-multi-3.4}) follows
 \begin{equation}\label{EB-multi-3.5}
 \begin{split}
\int_{\mathbb{R}^{3}}\phi(u)\Delta udx&=-\int_{\mathbb{R}^{3}}\phi'(u)|\nabla u|^{2}dx\\
&\leq C\int_{\mathbb{R}^{3}}(|u|^{p-1}+1)|\nabla u|^{2}dx\\
&\leq C\|u\|_{L^{6}}^{p-1}\|\nabla u\|_{L^{\frac{12}{7-p}}}^{2}+C\|\nabla u\|_{L^{2}}^{2}.
\end{split}
\end{equation}
In fact, applying Lemma \ref{lem-5} and Lemma \ref{lem-6}, we have
\begin{equation*}
\begin{split}
&\|u\|_{L^{6}}^{p-1}\leq C\|\nabla u\|_{L^{2}}^{p-1},\quad\quad\quad\hbox{and}
\\
&\|\nabla u\|_{L^{\frac{12}{7-p}}}^{2}\leq C[\|\nabla u\|_{L^{2}}^{\frac{5-p}{4}}\|\Delta u\|_{L^{2}}^{\frac{p-1}{4}}]^{2}\leq C\|\nabla u\|_{L^{2}}^{\frac{5-p}{2}}\|\Delta u\|_{L^{2}}^{\frac{p-1}{2}}.
\end{split}
\end{equation*}
Thus,
 \begin{equation}\label{EB-multi-3.6}
 \begin{split}
\int_{\mathbb{R}^{3}}\phi(u)\Delta udx&\leq C\|\nabla u\|_{L^{2}}^{\frac{p+3}{2}}\|\Delta u\|_{L^{2}}^{\frac{p-1}{2}}+C(u_{0})\\
&\leq C(\|\nabla u\|_{L^{2}}^{\frac{p+3}{2}})^{\frac{4}{5-p}}+\frac{1}{2}(\|\Delta u\|_{L^{2}}^{\frac{p-1}{2}})^{\frac{4}{p-1}}+C(u_{0})\\
&\leq C\|\nabla u\|_{L^{2}}^{\frac{2(p+3)}{(5-p)}}+\frac{1}{2}\|\Delta u\|_{L^{2}}^{2}+C(u_{0})\\
&\leq \frac{1}{2}\|\Delta u\|_{L^{2}}^{2}+C(u_{0}).
\end{split}
\end{equation}
Substituting \eqref{EB-multi-3.6} into \eqref{EB-multi-3.4}, we get
 \begin{equation}\label{EB-multi-3.7}
\frac{d}{dt}\|u\|_{L^2}^2+\|\Delta u\|_{L^{2}}^2\leq C(u_{0}).
\end{equation}
By the Gr\"{o}nwall inequality, we obtain
\begin{equation}
\label{EM-multi-3.8}
\|u\|_{L^2}^2+\int_{0}^{t}\|\Delta u\|_{L^{2}}^2 d\tau\leq C(T,u_{0}).
\end{equation}

On the other hand, multiplying the equations \eqref{EB-multi-1.1} by $(\Delta^{2} u)$ and then integrating in $x\in \mathbb{R}^3$, thanks to integration by parts and the Young's inequality, we get
\begin{equation}\label{EB-multi-3.9}
\begin{split}
\frac{1}{2}\frac{d}{dt}\|\Delta u\|^{2}&=\int_{\mathbb{R}^{3}}\Delta(-\Delta u+\phi(u))\Delta^{2}udx\\
&=-\|\Delta^{2} u\|^{2}+\int_{\mathbb{R}^{3}}\phi'(u)\Delta u\cdot\Delta^{2}udx+\int_{\mathbb{R}^{3}}\phi''(u)|\nabla u|^{2}\Delta^{2}udx.
\end{split}
\end{equation}
However, following from \eqref{EB-multi-1.4} and Young's inequality, there hold
\begin{equation}\label{EB-multi-3.10}
\begin{split}
&\int_{\mathbb{R}^{3}}\phi'(u)\Delta u\cdot\Delta^{2}udx\leq C\int_{\mathbb{R}^{3}}|(\phi'(u))\Delta u|^2dx+\frac{1}{7}\|\Delta^{2} u\|_{L^{2}}^{2}\\
\leq &C\int_{\mathbb{R}^{3}}|u|^{2(p-1)}|\Delta u|^{2}dx+C\|\Delta u\|_{L^{2}}^{2}+\frac{1}{7}\|\Delta^{2} u\|_{L^{2}}^{2}\\
\leq &C\big(\int_{\mathbb{R}^{3}}|u|^{2(p-1)\cdot\frac{4}{p-1}}dx\big)^{\frac{p-1}{4}}\big(\int_{\mathbb{R}^{3}}|\Delta u|^{2\cdot\frac{4}{5-p}}dx\big)^{\frac{5-p}{4}}+C\|\Delta u\|_{L^{2}}^{2}+\frac{1}{7}\|\Delta^{2} u\|_{L^{2}}^{2}\\
\leq &C\|u\|_{L^{8}}^{2(p-1)}\|\Delta u\|_{L^{\frac{8}{5-p}}}^{2}+C\|\Delta u\|_{L^{2}}^{2}+\frac{1}{7}\|\Delta^{2} u\|_{L^{2}}^{2}.
\end{split}
\end{equation}
In fact, thanks to the Lemma \ref{lem-6}, there hold
$$\|u\|_{L^{8}}\leq C\|\Delta u\|_{L^{2}}^{\frac{1}{8}}\|u\|_{L^{6}}^{\frac{7}{8}}\leq C\|\Delta u\|_{L^{2}}^{\frac{1}{8}}\|\nabla u\|_{L^{2}}^{\frac{7}{8}},$$
$$\|\Delta u\|_{L^{\frac{8}{5-p}}}\leq C\|\Delta^{2} u\|_{L^{2}}^{\frac{3(p-1)}{16}}\|\Delta u\|_{L^{2}}^{\frac{19-3p}{16}}.$$
Thanks to \eqref{EB-multi-3.3}, we have
\begin{equation}\label{EB-multi-3.11}
\begin{split}
&C\|u\|_{L^{8}}^{2(p-1)}\|\Delta u\|_{L^{\frac{8}{5-p}}}^{2}\\
\leq& C(\|\Delta u\|_{L^{2}}^{\frac{p-1}{4}}\|\nabla u\|_{L^{2}}^{\frac{7(p-1)}{4}})\|\Delta^{2} u\|_{L^{2}}^{\frac{3(p-1)}{8}}\|\Delta u\|_{L^{2}}^{\frac{19-3p}{8}}\\
\leq &C\|\Delta u\|_{L^{2}}^{\frac{17-p}{8}}\|\Delta^{2} u\|_{L^{2}}^{\frac{3(p-1)}{8}}\leq C\|\Delta u\|_{L^{2}}^{\frac{2(17-p)}{19-3p}}+\frac{1}{14}\|\Delta^{2} u\|_{L^{2}}^{2}\\
\leq& C\|\Delta u\|_{L^{2}}^{2}(\|\Delta u\|_{L^{2}}^{2}+1)+\frac{1}{14}\|\Delta^{2} u\|_{L^{2}}^{2}.
\end{split}
\end{equation}
Thus,
\begin{equation}\label{EB-multi-3.12}
\int_{\mathbb{R}^{3}}\phi'(u)\Delta u\cdot\Delta^{2}udx\leq C\|\Delta u\|_{L^{2}}^{2}(\|\Delta u\|_{L^{2}}^{2}+1)+\frac{3}{14}\|\Delta^{2} u\|_{L^{2}}^{2}.
\end{equation}
And following from \eqref{EB-multi-1.4}, Lemma \ref{lem-6} and Young's inequality, we can obtain
\begin{equation}\label{EB-multi-3.13}
\begin{split}
&\int_{\mathbb{R}^{3}}\phi''(u)|\nabla u|^{2}\cdot\Delta^{2}udx\leq C\int_{\mathbb{R}^{3}}|(\phi''(u))|\nabla u|^{2}|^2 dx+\frac{1}{7}\|\Delta^{2} u\|_{L^{2}}^{2}\\
\leq& C\int_{\mathbb{R}^{3}}|u|^{2(p-2)}|\nabla u|^{4}dx+C\|\nabla u\|^{4}_{L^{4}}+\frac{1}{7}\|\Delta^{2} u\|_{L^{2}}^{2}\\
\leq &C\big(\int_{\mathbb{R}^{3}}|u|^{2(p-2)\cdot\frac{4}{p-2}}dx\big)^{\frac{p-2}{4}}
\big(\int_{\mathbb{R}^{3}}|\nabla u|^{4\cdot\frac{4}{6-p}}dx\big)^{\frac{6-p}{4}}+C\|\Delta^{2} u\|_{L^{2}}\|\nabla u\|^{3}_{L^{2}}+\frac{1}{7}\|\Delta^{2} u\|_{L^{2}}^{2}\\
\leq& C\|u\|_{L^{8}}^{2(p-2)}\|\nabla u\|^{4}_{L^{\frac{16}{6-p}}}+\frac{3}{14}\|\Delta^{2} u\|_{L^{2}}^{2}+C\|\nabla u\|^{6}_{L^{2}},
\end{split}
\end{equation}
and
\begin{equation}\label{EB-multi-3.14}
\begin{split}
&C\|u\|_{L^{8}}^{2(p-2)}\|\nabla u\|^{4}_{L^{\frac{16}{6-p}}}
\leq C\big(\|\Delta u\|_{L^{2}}^{\frac{1}{8}}\|u\|_{L^{6}}^{\frac{7}{8}}\big)^{2(p-2)}\big(\|\Delta^{2} u\|_{L^{2}}^{\frac{2+p}{16}}\|\nabla u\|^{\frac{14-p}{16}}_{L^{2}}\big)^{4}\\
\leq& C\|\Delta u\|_{L^{2}}^{\frac{p-2}{4}}\| u\|_{L^{6}}^{\frac{7(p-2)}{4}}\|\Delta^{2} u\|_{L^{2}}^{\frac{2+p}{4}}\|\nabla u\|^{\frac{14-p}{4}}
\leq C\|\Delta u\|_{L^{2}}^{\frac{2(p-2)}{6-p}}+\frac{1}{14}\|\Delta^{2} u\|_{L^{2}}^{2}\\
\leq& C(\|\Delta u\|_{L^{2}}^{2}+1)^2+\frac{1}{14}\|\Delta^{2} u\|_{L^{2}}^{2},
\end{split}
\end{equation}
due to  Lemma \ref{lem-5} .
Therefore,
\begin{equation}\label{EB-multi-3.15}
\begin{split}
\int_{\mathbb{R}^{3}}\phi''(u)|\nabla u|^{2}\cdot\Delta^{2}udx\leq C\|\Delta u\|_{L^{2}}^{2}(\|\Delta u\|_{L^{2}}^{2}+1)+\frac{2}{7}\|\Delta^{2} u\|_{L^{2}}^{2}+C(u_{0}).
\end{split}
\end{equation}
Substituting \eqref{EB-multi-3.12},\eqref{EB-multi-3.15} into \eqref{EB-multi-3.9}, we get
\begin{equation}\label{EB-multi-3.16}
\frac{d}{dt}\|\Delta u\|_{L^{2}}^{2}+\|\Delta^{2} u\|_{L^{2}}^{2}\leq C\|\Delta u\|_{L^{2}}^{2}(1+\|\Delta u\|_{L^{2}}^{2})+C(u_{0}).
\end{equation}
Hence, we get from Gr\"{o}nwall's inequality and \eqref{EB-multi-3.9} that for all $t \in [0, T)$,
\begin{equation}\label{EB-multi-3.17}
\|\Delta u\|_{L^{2}}^{2}\leq (\|\Delta u_{0}\|_{L^{2}}^{2}+\int_{0}^{t}C(u_{0})d\tau)\exp\{C\int_{0}^{t}(1+\|\Delta u\|_{L^{2}}^{2})d\tau\}\leq C(u_{0},T).
\end{equation}
Integrating \eqref{EB-multi-3.16} with respect to the time $t$, we get for $t \in [0, T)$,
\begin{equation}
\label{3-17-1}
\|\Delta u\|_{L^{2}}^{2}+\int_{0}^{t}\|\Delta^{2} u\|_{L^{2}}^{2}d\tau\leq C(u_{0},T).
\end{equation}
It gives us that
\begin{equation}\label{EB-multi-3.18}
\begin{split}
u\in {\mathcal{C}}([0, T); H^2(\mathbb{R}^3))\cap L^{2}([0, T); H^4(\mathbb{R}^3)).
\end{split}
\end{equation}
At last, taking the $L^{2}$ inner product of the equations \eqref{EB-multi-1.1} with $\partial_{t}u$ ensures
\begin{equation}\label{EB-multi-3.19}
\begin{split}
\|\partial_{t}u\|_{L^{2}}^{2}\leq \|\Delta^{2}u\|_{L^{2}}^{2}+\int_{\mathbb{R}^3}|\Delta \phi(u)|^{2}dx+\frac{1}{2}\|\partial_{t}u\|_{L^{2}}^{2}.
\end{split}
\end{equation}
In fact, following the process of the calculuses \eqref{EB-multi-3.10}-\eqref{EB-multi-3.14}, we can easily get
\begin{equation}\label{EB-multi-3.20}
\begin{split}
&\int_{\mathbb{R}^3}|\Delta \phi(u)|^{2}dx\leq C\int_{\mathbb{R}^3}|\Delta u|^{2}|\phi'(u)|^{2}dx+C\int_{\mathbb{R}^3}|\nabla u|^{4}|\phi''(u)|^{2}dx\\
&\leq C\int_{\mathbb{R}^3}|\Delta u|^{2}|u|^{2(p-1)}dx+C\int_{\mathbb{R}^3}|\nabla u|^{4}|u|^{2(p-2)}dx+C\|\Delta u\|_{L^{2}}^{2}+C\|\nabla u\|_{L^{4}}^{4}\\
&\leq C\|\Delta u\|_{L^{2}}^{2}(\|\Delta u\|_{L^{2}}^{2}+1)+C\|\nabla u\|^{6}_{L^{2}}+\frac{1}{4}\|\Delta^{2} u\|_{L^{2}}^{2}+C
\end{split}
\end{equation}
Substituting \eqref{EB-multi-3.20} into \eqref{EB-multi-3.19}, and integrating \eqref{EB-multi-3.19} with respect to the time $t$, we get for $t \in [0, T)$, we obtain from \eqref{EB-multi-3.18},
$$\|\partial_{t}u\|_{L^{2}([0,T);L^{2})}^{2}\leq C(T,u_{0}).$$
which ends the proof of Lemma \ref{lem-3-1}.

\end{proof}

\renewcommand{\theequation}{\thesection.\arabic{equation}}
\setcounter{equation}{0}

\section{Local well-posedness}
We will prove the local well-posedness of the equations \eqref{EB-multi-1.1}.
\begin{thm}\label{thm-4-1}
Under the assumptions in Theorem \ref{thm-main-1}, there exist $T>0$ and a unique solution $u$ on $[0,T]$ of the Cahn-Hilliard equations \eqref{EB-multi-1.1} such that
\begin{equation}\label{EB-multi-4.1}
\begin{split}
u\in {\mathcal{C}}([0, T); H^{2}(\mathbb{R}^3))\cap L^{2}([0, T); H^{4}(\mathbb{R}^3)).
\end{split}
\end{equation}
\end{thm}

\begin{proof}
We are going to use the energy method to prove it in several steps.\\
{\bf Step 1: Construction of an approximate solution sequence.}

We shall first use the classical Friedrich¡¯s regularization method to construct the approximate solutions to \eqref{EB-multi-1.1}. In order to do so, let us define the sequence of frequency cut-off operators $(P_{n})_{n\in\mathbb{N}}$ by
$$P_{n}a\triangleq\mathcal{F}^{-1}(\mathbf{1}_{B(0,n)}\hat{a}).$$
and we define $u_{n}$ via
\begin{equation}\label{EB-multi-4.2}
 \left\{
    \begin{array}{l}
    \partial_{t}u_{n}=P_{n}\Delta\mu_{n},\quad \mu_{n}=P_{n}\phi(P_{n}u_{n})-\Delta P_{n}u_{n},\\
    u_{n}|_{t=0}=P_{n}u_{0}(x).
   \end{array}
    \right.
    \end{equation}
where $\mathbf{1}_{B(0,n)}$ is a characteristic function on the ball $B(0,n)$ centered at the origin with
radius $n$ with $n\in\mathbb{N}$.

Without loss of generality, we restrict $n>n_{0}$ in what follows, where we choose the integer $n_{0}$ so large that
$$(\sum_{|m|>n_{0}}|m|^{-2s})^{\frac{1}{2}}(\sum_{|m|>n_{0}}|m|^{2s}|\hat{u}_{0}|^{2})^{\frac{1}{2}}\leq c_{0},$$
which implies that for any $n>n_{0}$,
\begin{equation}\label{EB-multi-4.3}
\begin{split}
\|P_{n}u_{0}-u_{0}\|_{L^{\infty}}\leq(\sum_{|m|>n_{0}}|m|^{-2s})^{\frac{1}{2}}
(\sum_{|m|>n_{0}}|m|^{2s}|\hat{u}_{0}|^{2})^{\frac{1}{2}}\leq c_{0}.
\end{split}
\end{equation}

Because of properties of $L^{2}$ and $L^{1}$ functions, the Fourier transform of which are supported
in the ball $B(0,n)$, the system \eqref{EB-multi-4.2} appears to be an ordinary differential equation in the
space
$$L_{n}^{2}\triangleq\{a\in L^{2}(\mathbb{R}^{3}): Supp \ \hat{a}\subset B(0,n)\}.$$

Then the Cauchy-Lipschitz theorem allow us to deduce the existence of a local unique solution $u_{n}\in\mathcal{C}([0,T_{n}];L^{2}(\mathbb{R}^{3}))$ for the system \eqref{EB-multi-4.2}. Note that $P_{n}u_{n}$  is also a solution of \eqref{EB-multi-4.2}.  Thus the uniqueness of the solution implies that $P_{n}u_{n}=u_{n}$ and the solution $u_{n}$ is smooth.  Therefore, the approximate system \eqref{EB-multi-4.2} can be rewritten as
\begin{equation}\label{EB-multi-4.4}
 \left\{
    \begin{array}{l}
    \partial_{t}u_{n}=\Delta\mu_{n}, \mu_{n}=P_{n}\varphi(u_{n})-\Delta u_{n}\\
    u_{n}|_{t=0}=P_{n}u_{0}(x)
   \end{array}
    \right.
    \end{equation}
{\bf Step 2: Uniform estimates to the approximate solution}

Denote $T_{n}^{\ast}$ by the maximal existence time of the solution $u_{n}$, then, we first repeat the
argument in the proof of Lemma \ref{lem-3-1} to find
\begin{equation}\label{EB-multi-4.5}
\begin{split}
\|u_{n}\|_{L^{\infty}([0,T_{n}^{\ast})\times\mathbb{R}^{3})}\leq\sup_{\tau\in[0,T_{n}^{\ast})}\|u_{n}\|_{H^{2}(\mathbb{R}^{3})}
d\tau\leq C,
\end{split}
\end{equation}
where $C$ only depends on $P_{n}u_{0}$ and $T_{n}^{\ast}$.

Our goal in this step is to prove that there exists a positive time $T$ ($0<T<inf_{n\in\mathbb{N}}T_{n}^{\ast}$) such that
\begin{equation}\label{EB-multi-4.6}
\begin{split}
\|u_{n}\|_{L^{\infty}([0,T]\times\mathbb{R}^{3})}\leq C(T,u_{0}),
\end{split}
\end{equation}
and $u_{n}$  is uniformly bounded in the space
$$\mathcal{C}([0,T];H^{2}(\mathbb{R}^{3}))\cap L^{2}([0,T];H^{4}(\mathbb{R}^{3})).$$

The inequality \eqref{EB-multi-4.6} guarantees that the equations \eqref{EB-multi-4.2} is a regular problem, which plays a key role in what follows.

To obtain \eqref{EB-multi-4.6}, we consider \eqref{EB-multi-4.2} as a perturbation of its corresponding linear equations. For this, let's first define $u_{n}\triangleq u_{n}^{L}+\bar{u}_{n}$, where $u_{n}^{L}\triangleq e^{-t\Delta^{2}}P_{n}u_{0}$. Then we may rewrite \eqref{EB-multi-4.2} as the following $\bar{u}_{n}$ equations
\begin{equation}\label{EB-multi-4.7}
 \left\{
    \begin{array}{l}
    \partial_{t}\bar{u}_{n}+\Delta^{2}\bar{u}_{n}=f_{n}(u_{n}^{L},\bar{u}_{n}),\\
    \bar{u}_{n}|_{t=0}=0
   \end{array}
    \right.
    \end{equation}
with
\begin{equation}\label{EB-multi-4.8}
\begin{split}
f_{n}(u_{n}^{L},\bar{u}_{n})=\Delta P_{n}\phi(u_{n}^{L}+\bar{u}_{n}).
\end{split}
\end{equation}
From Remark \ref{rmk-2.1} and \eqref{EB-multi-4.3},  one can get, there is a positive time $T_{1}$ (independent of $n$) such
that
\begin{equation}\label{EB-multi-4.9}
\begin{split}
&\|u_{n}^{L}(t)\|_{L^{\infty}([0,T_{1}]\times\mathbb{R}^{3})}\leq \|u_{n}^{L}(t)-P_{n}u_{0}\|_{L^{\infty}([0,T_{1}]\times\mathbb{R}^{3})}+\|P_{n}u_{0}\|_{L^{\infty}(\mathbb{R}^{3})}\\
&\leq\|u_{n}^{L}(t)-P_{n}u_{0}\|_{L^{\infty}([0,T_{1}]\times\mathbb{R}^{3})}+\|P_{n}u_{0}-u_{0}\|_{L^{\infty}(\mathbb{R}^{3})}
+\|u_{0}\|_{L^{\infty}(\mathbb{R}^{3})}\leq M+2c_{0}
\end{split}
\end{equation}
for $\forall t\in[0,T_{1}]$. Moreover, it is easy to find
\begin{equation}\label{EB-multi-4.10}
\begin{split}
&\|u_{n}^{L}(t)\|_{L^{\infty}([0,+\infty);H^{s}(\mathbb{R}^{3}))}
+\|\Delta u_{n}^{L}(t)\|_{L^{2}([0,+\infty);H^{s}(\mathbb{R}^{3}))}\leq C\|u_{0}\|_{H^{s}(\mathbb{R}^{3})}.
\end{split}
\end{equation}
we again repeat the argument in the proof of Lemma \ref{lem-3-1} to find
\begin{equation}\label{EB-multi-4.11}
\begin{split}
\|\nabla \bar{u}_{n}\|_{L^{2}}\leq C(u_{0}),
\end{split}
\end{equation}
\begin{equation}\label{EB-multi-4.12}
\begin{split}
\frac{d}{dt}\|\bar{u}_{n}\|_{L^2}^2+2\|\Delta \bar{u}_{n}\|_{L^{2}}^2=\int_{\mathbb{R}^{3}}\phi(\bar{u}_{n}+u^{L}_{n})\Delta\bar{u}_{n}dx.
\end{split}
\end{equation}
\begin{equation}\label{EB-multi-4.13}
\begin{split}
\frac{d}{dt}\|\Delta\bar{u}_{n}\|_{L^2}^2+\|\Delta^{2} \bar{u}_{n}\|_{L^{2}}^2=\int_{\mathbb{R}^{3}}|\Delta\phi(\bar{u}_{n}+u^{L}_{n})|^{2}dx.
\end{split}
\end{equation}
And then, by the estimates (\ref{EB-multi-4.9}) and (\ref{EB-multi-4.10}),
there has
\begin{equation*}
\begin{split}
&\int_{\mathbb{R}^{3}}\phi(\bar{u}_{n}+u^{L}_{n})\Delta\bar{u}_{n}dx =-\int_{\mathbb{R}^{3}}\phi'(\bar{u}_{n}+u^{L}_{n})(|\nabla\bar{u}_{n}|^2+\nabla\bar{u}_{n}\cdot\nabla u^{L}_{n}) dx
\\
\leq &C\int_{\mathbb{R}^{3}}|\bar{u}_{n}|^{p-1}|\nabla\bar{u}_{n}|^{2}dx+
C\int_{\mathbb{R}^{3}}|\bar{u}_{n}|^{p-1}|\nabla\bar{u}_{n}\nabla u^{L}_{n}|dx\\
&+C(\|u^{L}_{n}\|_{L^{\infty}}^{p-1}+1)(\|\nabla\bar{u}_{n}\|_{L^{2}}^{2}+\|\nabla\bar{u}_{n}\|_{L^{2}}\|\nabla u^{L}_{n}\|_{L^{2}})\\
\leq &C\|\bar{u}_{n}\|_{L^{6}}^{p-1}\|\nabla\bar{u}_{n}\|_{L^{\frac{12}{7-p}}}^{2}+C(\|\bar{u}_{n}\|_{L^{6}}^{2(p-1)}
\|\nabla u^{L}_{n}\|_{L^{\frac{6}{4-p}}}^{2}+\|\nabla\bar{u}_{n}\|_{L^{2}}^{2})+C(u_{0})\\
\leq &C\|\nabla\bar{u}_{n}\|_{L^{2}}^{p-1}(\|\nabla\bar{u}_{n}\|_{L^{2}}^{\frac{5-p}{2}}
\|\Delta\bar{u}_{n}\|_{L^{2}}^{\frac{p-1}{2}})+C\|\nabla\bar{u}_{n}\|_{L^{2}}^{2(p-1)}(\|\Delta^{2} u^{L}_{n}\|_{L^{2}}^{\frac{p-1}{3}}\|\nabla u^{L}_{n}\|_{L^{2}}^{\frac{7-p}{3}}+1)+C(u_{0})\\
\leq& C(u_{0})+\frac{1}{4}\|\Delta \bar{u}_{n}\|_{L^{2}}^{2}.
\end{split}
\end{equation*}
Similarly, we could get
$$\int_{\mathbb{R}^{3}}|\Delta\phi(\bar{u}_{n}+u^{L}_{n})|^{2}dx\leq C(u_{0})+\frac{1}{4}\|\Delta^{2} \bar{u}_{n}\|_{L^{2}}^{2}.$$

Hence, we may claim that there is a positive time $T=T(u_{0})$ ($\leq$ min $\{T_{1}, T_{n}^{\ast}\}$) (with $T_{1}$ in Remark \ref{rmk-2.1}) independent of $n$, such that for all $n$,
\begin{equation}\label{EB-multi-4.14}
\begin{split}
\sup_{t\in[0,T]}\|\bar{u}_{n}\|_{H^{2}}^{2}\leq C(T;u_{0}),
\end{split}
\end{equation}
and thanks to Lemma \ref{lem-4}, there holds
\begin{equation}\label{EB-multi-4.15}
\begin{split}
\sup_{t\in[0,T]}\|\bar{u}_{n}\|_{L^{\infty}}\leq \sup_{t\in[0,T]}\|\bar{u}_{n}\|_{H^{2}}^{2}\leq C(T;u_{0}).
\end{split}
\end{equation}
Then, combining \eqref{EB-multi-4.15} with \eqref{EB-multi-4.9} yields \eqref{EB-multi-4.6}.
Integrating \eqref{EB-multi-4.11}-\eqref{EB-multi-4.13} with respect to the time $t$, we get for $\forall t \in [0,T]$,
$$\|\bar{u}_{n}\|_{H^{2}}^{2}+\int_{0}^{t}\|\Delta\bar{u}_{n}\|_{H^{2}}^{2}d\tau
\leq C(T;u_{0}).$$
Therefore,
\begin{equation}\label{EB-multi-4.16}
\begin{split}
\bar{u}_{n}\in\mathcal{C}([0,T];H^{2}(\mathbb{R}^{3}))\cap L^{2}([0,T];H^{4}(\mathbb{R}^{3})).
\end{split}
\end{equation}
Following from \eqref{EB-multi-4.10} and \eqref{EB-multi-4.16}, we have
\begin{equation}\label{EB-multi-4.17}
\begin{split}
u_{n}\in\mathcal{C}([0,T];H^{2}(\mathbb{R}^{3}))\cap L^{2}([0,T];H^{4}(\mathbb{R}^{3})).
\end{split}
\end{equation}

Moreover, we take the $L^{2}$ inner product of the equations \eqref{EB-multi-4.4} with $\partial_{t}u_{n}$ giving rise to
\begin{equation*}
\begin{split}
\frac{1}{2}\|\partial_{t}u_{n}\|_{L^{2}}^{2}&\leq \|\Delta\phi(u_{n})\|_{L^{2}}^{2}+\|\Delta^{2}u_{n}\|_{L^{2}}^{2}\\
&\leq C(\|\Delta u_{n}\|_{L^{2}}^{4}+\|\Delta u_{n}\|_{L^{2}}^{2}+\|\nabla u_{n}\|^{6}_{L^{2}}+\|\Delta^{2} u_{n}\|_{L^{2}}^{2}+1).
\end{split}
\end{equation*}
Then integrating with respect to the time $t$, we get
$$\|\partial_{t}u_{n}\|_{L^{2}([0,T];L^{2})}^{2}\leq C(T,u_{0}).$$
Therefore, we obtain
 \begin{equation}\label{EB-multi-4.18}
\{u_{n}\}_{n\in\mathbb{N}}\quad \mbox{ is uniformly bounded in }\quad {\mathcal{C}}([0, T]; H^2(\mathbb{R}^3)),
\end{equation}
 \begin{equation}\label{EB-multi-4.19}
 \{\Delta u_{n}\}_{n\in\mathbb{N}} \quad \mbox{ is uniformly bounded in }\quad L^2([0, T]; H^2(\mathbb{R}^3)),
 \end{equation}
\begin{equation}\label{EB-multi-4.20}
\{\partial_{t}u_{n}\}_{n\in\mathbb{N}} \quad \mbox{ is uniformly bounded in}\quad
L^2([0, T]; L^{2}(\mathbb{R}^3)).
\end{equation}

{\bf Step 3: Convergence}

Combining with the Aubin-Lions's compactness lemma and \eqref{EB-multi-4.18}, \eqref{EB-multi-4.19}, \eqref{EB-multi-4.20},  there exists a subsequence of $\{u_{n}\}_{n\in\mathbb{N}}$ (still denoted by $\{u_{n}\}_{n\in\mathbb{N}}$), which converges to some function $u\in\mathcal{C}([0,T];H^{2}(\mathbb{R}^{3}))\cap L^{2}([0,T];H^{4}(\mathbb{R}^{3}))$ such that
$$u_{n}\rightarrow u \quad \mbox{in} \quad L^2([0, T]; H^{4}(\mathbb{R}^3)).$$

Then passing to limit in \eqref{EB-multi-4.4}, it is easy to see that $u$ satisfies \eqref{EB-multi-1.1} in the weak sense. Moreover, there holds
\begin{equation}\label{EB-multi-4.21}
\begin{split}
\|u\|_{L^{\infty}([0,T];H^{2}(\mathbb{R}^{3}))}+\|\Delta u\|_{L^{2}([0,T];H^{2}(\mathbb{R}^{3}))}
+\|\partial_{t}u\|_{L^{2}([0,T];L^{2}(\mathbb{R}^{3}))}\leq C.
\end{split}
\end{equation}

{\bf Step 4:  Continuity in time of the solution}

Let's now prove the continuity in time of the solution.
For any $t \in [0,T]$ and $h$ such that $t + h \in [0,T]$, we deduce from \eqref{EB-multi-4.21} that
\begin{equation}\label{EB-multi-4.22}
\begin{split}
\|u(t+h)-u(t)\|_{L^{2}}^{2}\leq|h|^{2}\|\partial_{t}u\|_{L^{2}}^{2}+\frac{\varepsilon}{2}<\varepsilon
\end{split}
\end{equation}
for $|h|$ small enough. Hence, $u(t)$ is continuous in $L^{2}(\mathbb{R}^{3})$ for any time $t \in [0,T]$.

{\bf Step 5: Uniqueness of the solution}

First, let $u^1$ and $u^2$ be two solutions of \eqref{EB-multi-1.1} with the same initial data and satisfy \eqref{EB-multi-4.21}. We
denote $u^{1, 2}:=u^1-u^2$. Then $u^{1, 2}$ satisfies
\begin{equation*}
\left\{
    \begin{array}{l}
        \partial_t u^{1, 2}=\Delta[\phi(u^{1})-\phi(u^{2})]-\Delta^{2}u^{1, 2}\\
        u^{1, 2}|_{t=0}=0.
 \end{array}
    \right.
    \end{equation*}
Taking $L^2(\mathbb{R}^3)$ energy estimate, we have
\begin{equation}\label{EB-multi-4.23}
\begin{split}
 \frac{1}{2}\frac{d}{dt}\|u^{1, 2}\|_{L^2}^2&=\int_{\mathbb{R}^{3}}\Delta[\phi(u^{1})]u^{1, 2}dx-\int_{\mathbb{R}^{3}}\Delta[\phi(u^{2})]u^{1, 2}dx-\|\Delta u^{1, 2}\|_{L^{2}}^{2}\\
 &=\int_{\mathbb{R}^{3}}(\phi(u^{1})-\phi(u^{2}))\Delta u^{1, 2}dx-\|\Delta u^{1, 2}\|_{L^{2}}^{2}.
 \end{split}
\end{equation}
In fact, there exists a constant $\lambda$ ($0<\lambda<1$), such that
$$\phi(u^{1})-\phi(u^{2})=(u^{1}-u^{2})\phi'(\lambda u^{1}+(1-\lambda)u^{2}).$$
Then, we can obtain
\begin{equation}\label{EB-multi-4.24}
\begin{split}
\int_{\mathbb{R}^{3}}(\phi(u^{1})-\phi(u^{2}))\Delta u^{1, 2}dx&=\int_{\mathbb{R}^{3}}u^{1,2}\phi'(\lambda u^{1}+(1-\lambda)u^{2})\Delta u^{1, 2}dx\\
&\leq C\int_{\mathbb{R}^{3}}|u^{1, 2}||\Delta u^{1, 2}|(1+|\lambda u^{1}+(1-\lambda)u^{2}|^{p-1})dx\\
&\leq C\|u^{1, 2}\|_{L^{2}}\|\Delta u^{1, 2}\|_{L^{2}}(1+\|\lambda u^{1}+(1-\lambda)u^{2}\|_{L^{\infty}}^{p-1}).
 \end{split}
\end{equation}
 Owning to the Lemma \ref{lem-4} and \eqref{EB-multi-4.21}, we have
 $$\|\lambda u^{1}+(1-\lambda)u^{2}\|_{L^{\infty}}^{p-1}\leq C\|\lambda u^{1}\|_{L^{\infty}}^{p-1}+C\|(1-\lambda) u^{2}\|_{L^{\infty}}^{p-1}\leq C.$$
Thus,
\begin{equation}\label{EB-multi-4.25}
\begin{split}
\int_{\mathbb{R}^{3}}(\phi(u^{1})-\phi(u^{2}))\Delta u^{1, 2}dx&\leq C\|u^{1, 2}\|_{L^{2}}\|\Delta u^{1, 2}\|_{L^{2}}\\
&\leq C\|u^{1, 2}\|_{L^{2}}^{2}+\frac{1}{2}\|\Delta u^{1, 2}\|_{L^{2}}^{2}.
 \end{split}
\end{equation}
Substituting \eqref{EB-multi-4.25} into \eqref{EB-multi-4.23}, we obtain
\begin{equation*}
\frac{d}{dt}\|u^{1, 2}\|_{L^2}^2+\|\Delta u^{1,2}\|_{L^{2}}^{2}\leq C\|u^{1,2}\|_{L^{2}}^{2}.
\end{equation*}
Hence, it follows from Gr\"{o}nwall's inequality that  $u^{1, 2}(t)\equiv0$ for all $t \in [0, T]$.
The proof of Theorem \ref{thm-4-1} is completed.

\end{proof}

\renewcommand{\theequation}{\thesection.\arabic{equation}}
\setcounter{equation}{0}
\section{The global well-posedness}
We are now in a position to complete the proof of Theorem \ref{thm-main-1}.
\begin{proof}[Proof of Theorem \ref{thm-main-1}:]
Thanks to Theorem \ref{thm-4-1}, we conclude that: under the assumptions in
Theorem \ref{thm-main-1}, system \eqref{EB-multi-1.1} has a unique local solution $u$ satisfying \eqref{EB-multi-4.1}. Assume that $T^{\ast}>0$ is the maximal existence time of this solution, that is
$$u\in\mathcal{C}([0,T^{\ast});H^{2}(\mathbb{R}^{3}))\cap L^{2}([0,T^{\ast});H^{4}(\mathbb{R}^{3})).$$
It suffices to prove $T^{\ast}=+\infty$. We will argue by contradiction argument. Hence, we assume
$T^{\ast}<+\infty$ in what follows.

According to the basic energy estimates \eqref{EB-multi-3.3}, \eqref{EB-multi-3.7} and \eqref{3-17-1}, we get for $\forall t \in [0, T^{\ast})$,
\begin{equation}\label{EB-multi-5.1}
\|\nabla u\|_{L^2}^2\leq C(u_{0}).
\end{equation}
\begin{equation}\label{EB-multi-5.2}
\|u\|_{L^2}^2+\int_{0}^{t}\|\Delta u\|_{L^{2}}^2 d\tau\leq C(T^{\ast},u_{0}).
\end{equation}
\begin{equation}\label{EB-multi-5.3}
\|\Delta u\|_{L^2}^2+\int_{0}^{t}\|\Delta^{2} u\|_{L^{2}}^2 d\tau\leq C(T^{\ast},u_{0}).
\end{equation}
Which follows that
\begin{equation}\label{EB-multi-5.4}
\sup_{\tau\in[0,T^{\ast})}\|u\|_{H^{2}}^{2}\leq C(T^{\ast},u_{0})<+\infty.
\end{equation}
Applying Lemma \ref{lem-4}, we obtain
\begin{equation}\label{EB-multi-5.5}
\sup_{\tau\in[0,T^{\ast})}\|u\|_{L^{\infty}}^{2}\leq C\sup_{\tau\in[0,T^{\ast})}\|u\|_{H^{2}}^{2}\leq C(T^{\ast},u_{0})<+\infty.
\end{equation}
From this, the solution can be extended after $t = T^{\ast}$, which contradicts with the definition
of $T^{\ast}$. Hence, we get $T^{\ast} =+\infty$, and then complete the proof of Theorem \ref{thm-main-1}.

\end{proof}

\renewcommand{\theequation}{\thesection.\arabic{equation}}
\setcounter{equation}{0}
\section{A Polynomial Free Energy Density}
In the first few sections, we have got some results when  $\Phi$ and $ \phi$ satisfies Assumption \ref{aspt-1.1}. In this section, we're going to consider a special case of $\phi(u)= \sum_{i=0}^{4}a_{i}u^{i}$ ($a_{3}>0$ and $a_1<0$, $~i=1,2,3,4$), which has a polynomial free energy density. Evidently, the $\phi(u)$ also satisfies the conditions \eqref{EB-multi-1.4}, and we assume that $\Phi$ satisfies $\Phi(\cdot)\geq 0$, so the above conclusion also holds here. In fact, we can attain a better result in this special case.
\begin{aspt}\label{aspt-6.1}
Let $\phi(\cdot)\in \mathcal{C}(\mathbb{R}^{3})$ such that $\phi=\Phi'$ satisties
\begin{equation}\label{EB-multi-6.1}
\begin{split}
\phi(u)= \sum_{i=0}^{4}a_{i}u^{i},\ a_{3}>0,~a_{1}<0
 \end{split}
\end{equation}
where $a_{i}$($i=1,2,3,4$) are arbitrary constants.

\end{aspt}

In this paper, we intend to establish the global well-posedness of the system \eqref{EB-multi-1.1} with the $\phi$ in Assumption \ref{aspt-6.1}. Our main result is stated as follows.
\begin{thm}\label{thm-6-1}
Under Assumption \ref{aspt-6.1}, let $s>1$, $u_{0}\in H^{s}(\mathbb{R}^{3})$, then the system \eqref{EB-multi-1.1} has a unique global solution $u$ on $[0, +\infty)$ such that
\begin{equation*}
u\in {\mathcal{C}}([0, \infty);H^s(\mathbb{R}^3))\cap L_{loc}^2([0, \infty);H^{s+2}(\mathbb{R}^3)).
\end{equation*}
\end{thm}

\begin{rmk}\label{rmk-6.1}
According to the Theorem \ref{thm-main-1}, the existence and the uniqueness of the solution for the system \eqref{EB-multi-1.1} with the Assumption \ref{aspt-6.1} have proved. Then,  we just have to prove the solution $u$ also satisfies
\begin{equation*}
u\in {\mathcal{C}}([0, \infty);H^s(\mathbb{R}^3))\cap L^2([0, \infty);H^{s+2}(\mathbb{R}^3))
\end{equation*}
for any $s>1$.
\end{rmk}
\begin{proof}[Proof of Theorem \ref{thm-6-1}:]

As noted in the Remark \ref{rmk-6.1}, the equations \eqref{EB-multi-1.1} exist a unique solution $u$. Then, we only need prove that for any $s>1$,
\begin{equation*}
u\in {\mathcal{C}}([0, \infty);H^s(\mathbb{R}^3))\cap L^2([0, \infty);H^{s+2}(\mathbb{R}^3)).
\end{equation*}
In fact, applying the operator $\Lambda^{s}$ to the equations \eqref{EB-multi-1.1} and then taking the $L^{2}$ inner product with $\Lambda^{s}u$, we obtain
\begin{equation}\label{EB-multi-6.2}
\begin{split}
\frac{1}{2}\frac{d}{dt}\|u\|_{H^{s}}^{2}+\|\Delta u\|_{H^{s}}^{2}=\int_{\mathbb{R}^{3}}\Lambda^{s}\Delta(\sum_{i=0}^{4}a_{i}u^{i})\Lambda^{s}udx.
 \end{split}
\end{equation}
If $i=0$, $\int_{\mathbb{R}^{3}}\Lambda^{s}\Delta(a_{0})\Lambda^{s}udx=0$, so
\begin{equation}\label{EB-multi-6.3}
\begin{split}
\int_{\mathbb{R}^{3}}\Lambda^{s}\Delta(\sum_{i=0}^{4}a_{i}u^{i})\Lambda^{s}udx&=
\int_{\mathbb{R}^{3}}\Lambda^{s}(\sum_{i=1}^{4}a_{i}u^{i})\Delta\Lambda^{s}udx\\
&\leq C\|\sum_{i=1}^{4}a_{i}u^{i}\|_{H^{s}}^{2}+\frac{1}{2}\|\Delta u\|_{H^{s}}^{2}.
 \end{split}
\end{equation}
Thanks to Lemma \ref{lem-7} and \eqref{EB-multi-5.4}, \eqref{EB-multi-5.5}, we have
\begin{equation}\label{EB-multi-6.4}
\begin{split}
\|\sum_{i=1}^{4}a_{i}u^{i}\|_{H^{s}}&\leq C(1+\|u\|_{L^{\infty}})^{\sigma}\|(\sum_{i=1}^{4}a_{i}u^{i})''\|_{W^{\sigma,\infty}}\|u\|_{H^{s}}\\
&\leq C\|u\|_{H^{s}},
 \end{split}
\end{equation}
where $s>1$, and $\sigma>0$ is the smallest integer such that $\sigma>s$.
Inserting \eqref{EB-multi-6.3}, \eqref{EB-multi-6.4} into \eqref{EB-multi-6.2}, we obtain
\begin{equation}\label{EB-multi-6.5}
\begin{split}
\frac{d}{dt}\|u\|_{H^{s}}^{2}+\|\Delta u\|_{H^{s}}^{2}\leq C\|u\|_{H^{s}}^{2},
 \end{split}
\end{equation}
 Owing to the Gr\"{o}nwall inequality, we get
 $$\|u\|_{H^{s}}^{2}\leq C(u_{0}).$$
 Integrating \eqref{EB-multi-6.5} with respect to the time $t$, we get for $\forall t \in [0, T)$ $(T<\infty)$
\begin{equation}\label{EB-multi-6.6}
\begin{split}
\|u\|_{H^{s}}^{2}+\int_{0}^{t}\|\Delta u\|_{H^{s}}^{2}d\tau\leq C(u_{0},T),
 \end{split}
\end{equation}
which implies
$$u\in {\mathcal{C}}([0, T];H^s(\mathbb{R}^3))\cap L^2([0, T];H^{s+2}(\mathbb{R}^3)).$$

Next, repeating the proof of the Theorem \ref{thm-main-1}, we can get  $T^{\ast}=+\infty$ ($T^{\ast}$ is the
maximal existence time of this solution, such that $u\in {\mathcal{C}}([0, T^{\ast});H^s(\mathbb{R}^3))\cap L^2([0, T^{\ast});H^{s+2}(\mathbb{R}^3))$). And then we complete the proof of Theorem \ref{thm-6-1}.

\end{proof}

\vskip 0.2cm

\noindent {\bf Acknowledgments.} The work of Zhenbang Li is supported in part by NSF of China Grant 11801443 and ShaanXi province Department of Education Fund 15JK1347.

\end{document}